\documentclass[12pt]{article}
\usepackage{amsfonts}
\usepackage{amsmath}
\usepackage{mathrsfs}
\usepackage{amssymb}
\usepackage{amsthm}
\usepackage{fullpage}
\usepackage{lmodern}
\RequirePackage[affil-sl]{authblk}
\usepackage{enumitem}
\usepackage[color=green!40]{todonotes}
\setlist[description]{leftmargin=\parindent,labelindent=\parindent}
\RequirePackage{etoolbox}
\makeatletter
\patchcmd{\@maketitle}{center}{flushleft}{}{}
\patchcmd{\@maketitle}{center}{flushleft}{}{}
\patchcmd{\@maketitle}{\LARGE}{\Huge\sf}{}{}
\def\maketitle{{%
  
  \AB@maketitle}}
\makeatother
\theoremstyle{definition}

\newtheorem{construction}{Construction}[section]

\theoremstyle{plain}
\newtheorem{conjecture}[construction]{Conjecture}
\newtheorem{theorem}[construction]{Theorem}
\newtheorem{proposition}[construction]{Proposition}

\theoremstyle{remark}

\def \cF {{\cal F}}

\def \cH {{\cal H}}

\begin{document}
\title{\textit{\textbf{\boldmath Complex uniformly resolvable decompositions of $K_v$ }}}

\author[1]{\textsf{\textbf{Csilla Bujt\' as}}%
\thanks{Supported by the Slovenian Research Agency under the project N1-0108}$^,\!\! $
}
\author[2]{\textsf{\textbf{Mario Gionfriddo}}%
}
\author[2]{\textsf{\textbf{Elena Guardo}}%
	}
\author[$\!\! $ ]{\textsf{\textbf{\framebox{Lorenzo Milazzo}}}%
	}
\author[2]{\\\textsf{\textbf{Salvatore Milici}}%
	\thanks{Supported by MIUR and I.N.D.A.M. (G.N.S.A.G.A.), Italy and  by  Universit\`a degli Studi di
Catania, ``Piano della Ricerca 2016/2018 Linea di intervento 2'' }$^,\!\! $
}
\author[4]{\textsf{\textbf{Zsolt Tuza}}%
\thanks{Supported by the National Research,
 Development and Innovation Office -- NKFIH under the grant SNN 116095.}$^,
  $\thanks{corresponding author}$^{,3,} $}

\affil[1]{ \small Faculty of Mathematics and Physics, University of Ljubljana, Jadranska 19, Ljubljana 1000, Slovenia; \
	 e-mail: bujtas@fmf.uni-lj.si
}
\affil[2]{ \small Dipartimento di Matematica e Informatica,
Universit\`a di Catania,
 Viale A. Doria, 6, \hfill\break 95125 - Catania,  Italia; \
 e-mail: \{gionfriddo,guardo,milici\}@dmi.unict.it
}
\affil[3]{ \small Alfr\'ed R\'enyi Institute of Mathematics,  1053 Buda\-pest, Re\'altanoda u.~13--15, Hungary
\hfill\break $^4$ \small Department of Computer Science and Systems Technology, University of Pannonia,
 \hfill\break 8200 Veszpr\'em, Egyetem u.~10, Hungary; \
	 e-mail: tuza@dcs.uni-pannon.hu
}

\date{}

\maketitle

\begin{center}
\large {\vspace{-5ex} \em We dedicate this paper to the good friend and colleague Lorenzo Milazzo\\
who passed away in March 2019.}
\end{center}

\bigskip

\begin{abstract}
In this paper we consider the complex uniformly resolvable decompositions of the complete graph $K_v$ into subgraphs such that each resolution class contains only blocks isomorphic to the same graph from a given set $\cH$. We completely determine the spectrum for the cases
$\cH =\{K_2, P_3, K_3\}$, $\cH =\{P_4, C_4\}$, and $\cH =\{K_2, P_4, C_4\}$.
\end{abstract}

\bigskip



\bigskip

\section{Introduction and definitions}\label{introduzione}

Given a set $\cH$ of  pairwise non-isomorphic graphs, an \emph{$\cH$-decomposition} (or  \emph{$\cH$-design})
of a graph $G$ is a decomposition of the edge set of $G$ into
subgraphs (called \emph{blocks}) isomorphic to some element of $\cH$. An $\cH$-\emph{factor} of $G$ is a spanning subgraph of $G$
which is a vertex-disjoint union of some copies of graphs belonging to $\cH$.
If $\cH= \{H\}$, we will briefly speak of an $H$-factor. 
An $\cH$-decomposition of $G$ is \emph{resolvable} if its blocks can be partitioned into $\cH$-factors ($\cH$-\emph{factorization} or resolution of $G$). 
An $\cH$-factor in an $\cH$-factorization is referred to as a \emph{parallel class}.
Note that the parallel classes are mutually edge-disjoint, by definition.

An $\cH$-factorization $\cal F$ of $G$ is called \emph{uniform} if  each factor of ${\cal F}$ is an $H$-factor for some graph $H \in {\cal H}$.
A $K_2$-factorization of $G$ is known as a 1-\emph{factorization} and its factors are called 1-{\em factors}; it is well known that a 1-factorization of $K_v$  exists if and only if $v$ is even (\cite{Lu}).
If $\cH=\{F_1,\dots,F_k\}$ and $r_{i}\geq 0$ for $i=1,\dots,k$, we denote by $(F_1,\dots,F_k)$-URD$(v; r_{1},\dots,r_{k})$ a uniformly resolvable decomposition of the complete graph $K_v$ having exactly $r_i$ $F_{i}$-factors. 
A {\em complex} $(F_1,\dots,F_k)$-URD$(v; r_{1},\dots,r_{k})$ is a uniformly resolvable decomposition of the complete graph $K_v$ into $r_{1}+\dots+r_{k}$ parallel classes with the requirement that at least one parallel class is present for each $F_{i}\in\cH$, i.e., $r_{i}>0$ for $i=1,\dots, k$.

Recently, the existence problem for $\cH$-factorizations of $K_v$ has been studied and a lot of results have been obtained, especially on the following types of  uniformly resolvable ${\cal H}$-decompositions: for a set $\cH$ consisting of two complete graphs of orders at most five in \cite{DLD, R, SG, WG};
for a set $\cH$ of two or three paths on two, three, or four vertices in \cite{GM1,GM2, LMT}; for $\cH =\{P_3, K_3+e\}$ in \cite{GM}; for $\cH =\{K_3, K_{1,3}\}$ in \cite{KMT}; for $\cH =\{C_4, P_{3}\}$ in \cite{M}; for $\cH =\{K_3, P_{3}\}$ in \cite{MT}; for $1$-factors and $n$-stars in \cite{KKMT}; and  for $\cH =\{P_2, P_{3}, P_4\}$ in \cite{LMT}.
In connection with our current studies the following cases are most relevant:
  \begin{itemize}
    \item perfect matchings and parallel classes of triangles or 4-cycles
   ($\{K_2,K_3\}$ or $\{K_2,C_4\}$, Rees \cite{R});
    \item perfect matchings and parallel classes of 3-paths
   ($\{K_2,P_3\}$, Bermond et al.\ \cite{BHY}, Gionfriddo and Milici~\cite{GM1});
    \item parallel classes of 3-paths and triangles
   ($\{K_3,P_3\}$, Milici and Tuza \cite{MT}).
  \end{itemize}
In this paper we give a complete characterization of the spectrum (the set of all admissible combinations of the parameters) for the following two triplets of graphs and for the pair contained in one of them which is not covered by the cases known so far:
\begin{itemize}
   \item complex $\{K_2, P_3, K_3\}$-decompositions of order $v$ (Section \ref{sec:MPT}, Theorem \ref{theorem11});
   \item complex $\{K_2, P_4, C_4\}$-decompositions of order $v$ (Section \ref{sec:MPC},  Theorem \ref{theorem12});
   \item complex $\{P_4, C_4\}$-decompositions of order $v$ (Section \ref{sec:PC}, Theorem \ref{theorem13}).
 \end{itemize}
We summarize the formulation of those results in the concluding section, where a conjecture related to the method of ``metamorphosis'' of parallel classes is also raised.
We provide the basis for this approach by applying linear algebra in Section \ref{sec:2}.

\section{Local metamorphosis}
\label{sec:2}

In this section we prove three relations between uniform parallel classes of 4-cycles and 4-paths that will be used in the proofs of our main theorems. Before presenting the new statements, let us recall  the  Milici--Tuza--Wilson Lemma from \cite{MT}.

\begin{theorem}
\label{theorem1}\cite{MT}
The union of two 
parallel classes of\/ $3$-cycles of\/ $K_v$ can be decomposed into three
 parallel classes of\/ $P_3$.
\end{theorem}

The next two results, Theorems~\ref{theorem2} and \ref{theorem3}, will directly imply Theorem~\ref{theorem4} which states a pure metamorphosis from $4$-cycles to $4$-paths.

\begin{theorem}
 \label{theorem2} The union of two parallel classes of\/ $C_4$ is decomposable into two parallel classes of\/ $P_4$ and one perfect matching.
\end{theorem}

 \begin{proof}
Let the vertices be $v_1,\dots,v_n$ where $n$ is a multiple of 4.
The union of two parallel classes of $C_4$ forms a 4-regular graph $G$ with $2n$ edges, say $e_1,\dots,e_{2n}$.
We associate a Boolean variable $x_i$ with each edge $e_i$ ($1\le i\le 2n$) and construct a system of linear equations over $GF(2)$, which
 has $\frac{3n}{2} -1$ equations over the $2n$ variables.
Let us set
 $$
   x_{i_1} + x_{i_2} + x_{i_3} + x_{i_4} = 1
     \qquad (\mbox{\rm mod } 2)
 $$
  for each 4-tuple of indices such that
 $e_{i_1},e_{i_2},e_{i_3},e_{i_4}$ are either the edges of
  a $C_4$ in a parallel class (call this a $C$-equation) or are
  the four edges incident with a vertex $v_i$ (a $V$-equation).
This gives $\frac 32 n$ equations, but the $V$-equation
 for $v_n$ can be omitted since the $n$ $V$-equations sum up
 to 0 (as each edge is counted twice in the total sum)
 and therefore the one for $v_n$ follows from the others.

We claim that this system of equations is contradiction-free
 over $GF(2)$.
To show this, we need to prove that if the left sides of a
 subcollection 
 $\mathcal{E}$ of the equations sum up to 0,
 then also the right sides have zero sum; that is, the number
  $|\mathcal{E}|$ of its equations is even.

Observe that each variable is present in precisely three equations:
 in one $C$-equation and two $V$-equations.
Hence, to have zero sum on the left side, any $x_i$ should either
 not appear in any equations of $\mathcal{E}$ or be present in
 precisely two.
This means one of the following two situations.
 \begin{description}

  \item[$(T1)$] If $(e_{i_1},e_{i_2},e_{i_3},e_{i_4})$ is a 4-cycle
   (in this cyclic order of edges)
   and its $C$-equation belongs to $\mathcal{E}$, then precisely two
   related $V$-equations must be present in $\mathcal{E}$, namely either
   those for the vertices $e_{i_1}\cap e_{i_2}$ and $e_{i_3}\cap e_{i_4}$
   or those for $e_{i_2}\cap e_{i_3}$ and $e_{i_1}\cap e_{i_4}$.

  \item[$(T2)$] If $(e_{i_1},e_{i_2},e_{i_3},e_{i_4})$ is a 4-cycle
   such that its $C$-equation does not belong
   to $\mathcal{E}$ but some $x_{i_j}$ ($1\le j\le 4$) is
   involved in $\mathcal{E}$, then all the four $V$-equations for
   $v_{i_1},v_{i_2},v_{i_3},v_{i_4}$ must be present in $\mathcal{E}$.
 \end{description}

In the first and second parallel class of 4-cycles, respectively,
 let us denote the number of cycles of type $(T1)$ by $a_1$ and $b_1$,
 and that of type $(T2)$ by $a_2$ and $b_2$.
Then the number of $V$-equations
 in $\mathcal{E}$ is equal to both $2a_1+4a_2$ and $2b_1+4b_2$,
 which is the same as the average of these two numbers.
Thus, the number $|\mathcal{E}|$ of equations is equal to
 $$
   a_1 + b_1 +  \frac 12 ((2a_1+4a_2) + (2b_1+4b_2))
     = 2 ( a_1+a_2 + b_1+b_2 )
 $$
 that is even, as needed.

Since the system of equations is non-contradictory, it has a solution
 $\xi\in\{0,1\}^{2n}$ over $GF(2)$.
We observe further that in any $C$-equation
 $x_{i_1} + x_{i_2} + x_{i_3} + x_{i_4} = 1$
 we may switch the values from $\xi(x_{i_j})$
  to $1-\xi(x_{i_j})$ simultaneously for all
 $1\le j\le 4$, and doing so the modified values remain a solution
 because the parities of sums in the $V$-equations do not change either.
In this way, we can transform $\xi$ to a basic solution $\xi_0$
 in which every $C$-equation contains precisely one 1 and three 0s.
Since each $V$-equation contains precisely two or zero variables
 from each $C$-equation, it follows that in the basic solution $\xi_0$
 each $V$-equation, too, contains precisely one $1$ and three $0$s.

As a consequence, the variables which have $x_i=1$ in the basic
 solution define a perfect matching in $G$
 (since at most two 1s may occur at each vertex, and then the
 corresponding $V$-equation implies that there is precisely one).
Moreover, removing those edges from $G$, each cycle of each parallel
 class becomes a $P_4$.
In this way we obtain two parallel classes of $P_4$, and one further
 class which is a perfect matching.
\end{proof}

\begin{theorem}
 \label{theorem3}The edge-disjoint union of a perfect matching and a parallel class
  of\/ $C_4$ is decomposable into two parallel classes of\/ $P_4$.
\end{theorem}

\begin{proof}
We apply several ideas from the previous proof, but in a somewhat
 different way.
We now introduce Boolean variables $x_1,\dots,x_n$ for the edges $e_1,\dots,e_n$
 of the 4-cycles only; but still there will be two kinds of linear
 equations, namely $n/4$ of them for 4-cycles (called $C$-equations)
 and $n/2$ of them for the edges of the perfect matching ($M$-equations).
They are of the same form as before:
 $$
   x_{i_1} + x_{i_2} + x_{i_3} + x_{i_4} = 1
     \qquad (\mbox{\rm mod } 2)
 $$
The $C$-equations require $e_{i_1},e_{i_2},e_{i_3},e_{i_4}$ to be
 the edges of a $C_4$ in the parallel class.
The $M$-equations take $e_{i_1},e_{i_2},e_{i_3},e_{i_4}$ as the
 four edges incident with a matching edge.
If a matching edge is the diagonal of a 4-cycle, then their
 equations coincide; and if a matching edge shares just one
 vertex with a 4-cycle then the $M$-equation and the
 $C$-equation share two variables which correspond to
 consecutive edges on the cycle.
Further, we recall that the matching is edge-disjoint from
 the cycles, therefore each variable associated with a
 cycle-edge occurs in precisely two distinct $M$-equations.

These facts imply that only two types of $C$-equations can
 occur in a subcollection $\mathcal{E}$ of equations
 whose left sides sum up to 0 over $GF(2)$.
 \begin{description}

  \item[$(T1)$] If $(e_{i_1},e_{i_2},e_{i_3},e_{i_4})$ is a 4-cycle
   and its $C$-equation belongs to $\mathcal{E}$, then the
   corresponding $C_4$ has precisely two (antipodal) vertices
   for which the $M$-equations of the incident matching edges are present in $\mathcal{E}$.
  (At the moment it is unimportant whether those two
   vertices form a matching edge or not.)

  \item[$(T2)$] If $(e_{i_1},e_{i_2},e_{i_3},e_{i_4})$ is a 4-cycle
    whose $C$-equation does not belong
   to $\mathcal{E}$ but some $x_{i_j}$ ($1\le j\le 4$) is
   involved in $\mathcal{E}$, then each $M$-equation belonging to a
   matching edge incident with some of the four vertices
   $v_{i_1},v_{i_2},v_{i_3},v_{i_4}$ is present in $\mathcal{E}$.
  (It is again unimportant whether one or both or none of the diagonals
   of the $C_4$ in question is a matching edge.)
 \end{description}

Let now $a_1$ and $a_2$ denote the number of cycles with type $(T1)$
 and type $(T2)$, respectively.
By what has been said, the number of vertices requiring an
 $M$-equation is equal to $2a_1+4a_2$.
Since each of those equations is now counted at both ends of the
 corresponding matching edge, we obtain that $\mathcal{E}$ contains
 exactly $a_1+2a_2$ $M$-equations; moreover it has $a_1$ $C$-equations,
 by definition.
Thus, the number $|\mathcal{E}|$ of equations is equal to
 $a_1 + (a_1+2a_2) = 2 ( a_1+a_2 )$
 which is even.
Thus, if the left sides in $\mathcal{E}$ sum up to zero, then
 also the right sides have sum 0 in $GF(2)$. It proves that
 the system of the $3n/4$ equations is contradiction-free and has a solution
 $\xi\in\{0,1\}^n$ over $GF(2)$.
 
Now, we observe that in any $C$-equation
 $x_{i_1} + x_{i_2} + x_{i_3} + x_{i_4} = 1$ we may switch the
 values from $\xi(x_{i_j})$ to $1-\xi(x_{i_j})$ simultaneously for all
 $1\le j\le 4$. Doing so, the modified values remain a solution
 as the parities of sums in the $M$-equations do not change either.
In this way we can transform $\xi$ to a basic solution $\xi_0$
 in which every $C$-equation contains precisely one $1$ and three $0$s.
Since each $M$-equation has precisely two or four or zero variables
 from any $C$-equation, it follows that in the basic solution $\xi_0$
 each $M$-equation, too, contains precisely one $1$ and three $0$s.

As a consequence, the variables (cycle-edges) which have $x_i=1$ in the basic
 solution establish a pairing between the edges of the
 original matching.
Hence the set $I=\{e_i : \xi_0(x_i)=1\}$ together with the edges of the given matching factor determines a $P_4$-factor.
Moreover, removing the edges of $I$ from the $4$-cycles, we obtain another parallel class of paths $P_4$.
 \end{proof}

These two types of metamorphosis can be combined to obtain the following third one.

\begin{theorem}
 \label{theorem4} The union of three parallel classes of\/ $C_4$ is
  decomposable into four parallel classes of\/ $P_4$.
\end{theorem}

 \begin{proof}

Applying Theorem \ref{theorem2} we transform the union of the first and the second $C_4$-class into two $P_4$-classes and a perfect matching.
After that we combine the third $C_4$-class with the perfect matching just obtained into two further $P_4$-classes, by Theorem \ref{theorem3}.
 \end{proof}

\section{The spectrum for $\cH=\{K_2,P_3,K_3\}$}
\label{sec:MPT}

In this section we consider complex uniformly resolvable decompositions of the complete graph $K_v$ into $m$ classes containing only copies of 1-factors (perfect matchings), $p$ classes containing only copies of paths $P_3$ and  $t$ classes containing only copies of triangles $K_3$.
 The current problem is to determine the set of feasible
 triples $(m,p,t)$ such that $m\cdot p\cdot t \ne 0$, for which
 there exists a complex $(K_2,P_3,K_3)$-URD$(v;m,p,t)$.
A little more than that, for $v=6$ and $v=12$ we shall list
 also those feasible $(m,p,t)$ in which $m$ or $p$ or $t$ is zero.

\begin{theorem}
 \label{theorem11}
 The necessary and sufficient conditions for the existence of a
  complex\/ $(K_2,P_3,K_3)$-URD\/$(v;m,p,t)$ are:
   \begin{description}
   \item[$(i)$] $v\ge 12$ and\/ $v$ is a multiple of\/ $6$;
   \item[$(ii)$] $3m + 4p + 6t = 3v-3$.
  \end{description}
Moreover, the parameters\/ $m,p,t$ are in the following ranges:
   \begin{description}
   \item[$(iii)$] $1\le m\le v-7$ and\/ $m$ is odd,

    \qquad \ $3\le p\le 3\cdot\lfloor \frac v4 -1 \rfloor$,

    \qquad \ $1\le t\le \lfloor \frac v2 -3 \rfloor$.
  \end{description}
\end{theorem}

\begin{proof}
We first prove that the conditions are necessary. Divisibility of $v$ by 6 is immediately seen, due to the presence of $K_3$-classes and $1$-factors.
We observe further that the number of edges in a parallel class is $v$ for a triangle-class, $2v/3$ for a $P_3$-class and $v/2$ for a matching.
Thus, in any $(K_2,P_3,K_3)$-URD$(v;m,p,t)$ we must have
 $$
   \frac{mv}{2} + \frac{2pv}{3} + tv = {v\choose 2} .
 $$
Dividing it by $v/6$, the assertion of $(ii)$ follows.

As  $(ii)$ implies,  $p$ is a multiple of 3, say $p=3x$. Then we obtain
$$ m+4x+2t=v-1$$
and also conclude that $m$ is odd.
Since $m\ge 1$, $t\ge 1$, and $x\ge 1$, 
 this equation yields
  $$
    m \leq v-7 , \qquad x \leq v/4 - 1 , \qquad t \leq v/2 - 3,
  $$
 implying the conditions listed in $(iii)$, and the
  first one also excludes $n=6$.
This completes the proof that the conditions $(i)$--$(iii)$ are necessary.

To prove the sufficiency of $(i)$--$(ii)$, we consider $v\ge 18$ first.
Since $v$ is a multiple of 6 according to $(i)$, there exists a Nearly Kirkman Triple System of order $v$, which means $m=1$ perfect
 matching and $t=\frac n2 -1$ parallel classes of triangles.
More generally, for every odd $m$ in the range $1\le m\le v-7$,
 there exists a collection of $m$ perfect matchings and
 $t=\frac{v-1-m}2$ parallel classes of triangles, which together
 decompose $K_v$; this was proved in \cite{R}.
From such a system, for every $0<x<\frac{v-1-m}{4}$, we can take $2x$ parallel classes of triangles.
 Applying  Theorem \ref{theorem1} \cite{MT},
 also proved independently by Wilson (unpublished), we obtain $3x$ parallel classes of paths $P_3$. This gives a complex
$(K_2,P_3,K_3)$-URD$(v;m,3x,\frac{v-1-m}{2}-2x)$.
 For $n=12$, the statement  follows by  Proposition \ref{lemmaF2} below.
\end{proof}

\subsection{Small cases}

Obviously, the proofs of the necessary conditions that $v$ is a multiple of $6$,
 and that the equality $3m+4p+6t=3v-3$ must be satisfied by every $(K_2, P_3, K_{3})$-URD$(v;m,p,t)$,
 do not use the assumption $m\cdot p\cdot t\neq 0$.

\begin{proposition}
\label{lemmaF1}   There exists a\/ $(K_2, P_3, K_{3})$-URD\/$(6;m,p,t)$ if and only if\/ $(m,p,t)\in \{(5,0,0), (3,0,1),(1,3,0)\}$.
\end{proposition}

\begin{proof}
Putting $v=6$, the equation $3m+4p+6t=15$ has exactly four solutions $(m,p,t)$ over the nonnegative integers.
The case $(1,0,2)$ would correspond to an NKTS(6) which is known not to exist
\cite{R}. The case $(5,0,0)$ corresponds to a $1$-factorization of the complete  graph $K_{6}$ which is known to exist \cite{Lu}.
The case of $(1,3,0)$ is just the same as a $(K_2, P_3)$-URD$(6;1,3)$ that is known to exist \cite{GM2}. To see the existence for  $(3,0,1)$, consider  $V(K_{6})=\mathbb{Z}_{6}$ and the following classes:
$\{\{1,4\}, \{2,5\}, \{3,6\}\}$,  $\{\{1,5\}, \{2,6\}, \{3,4\}\}$, $\{\{1,6\}, \{2,4\}, \{3,5\}\}$,
$\{(1,2,3), (4,5,6)\}$.
\end{proof}

\begin{proposition}
\label{lemmaF2}   There exists a\/ $(K_2, P_3, K_{3})$-URD\/$(12;m,p,t)$ if and only if\/ $(m,p,t)\in \{(11,0,0), (9,0,1), (7,3,0), (7,0,2),(5,0,3), (5,3,1), (3,6,0),(3,3,2), (3,0,2),(1,6,1), (1,3,3)\}$.
\end{proposition}

\begin{proof}
Checking the nonnegative integer solutions of $3m+4p+6t=33$,
 the case of $(1,0,5)$ would correspond to an NKTS(12) which is known not to exist
\cite{R}. The case of $(11,0,0)$ corresponds to a 1-factorization of the complete graph $K_{12}$ that is known to exist \cite{CD}.
The result for the cases $(9,0,1)$, $(7,0,2)$, $(5,0,3)$, $(3,0,4)$ follows by \cite{R}. Applying  Theorem \ref{theorem1} to  $(7,0,2)$, $(5,0,3)$, $(3,0,4)$, we obtain the existence for $(7,3,0)$, $(5,3,1)$, $(3,3,2)$, and $(3,6,0)$.  The existence for the case $(1,3,3)$ is shown by the following construction. Let $V(K_{12})=\mathbb{Z}_{12}$, and consider the following parallel classes:
 \begin{itemize}
   \item matching: $\{\{1,6\}, \{2,4\}, \{3,0\},\{5,11\}, \{7,9\}, \{8,10\}\}$;
   \item paths: $\{\{5,0,11\}, \{8,1,7\}, \{9,2,3\},\{6,4,10\}\}$, $\{\{1,3,8\},\{7,5,10\},\{4,9,0\}, \{2,11,6\}\},\\\{ \{ 0,6,5\},\{2,7,4\},\{9,8,11\},\{3,10,1\}\}$;
   \item triangles: $\{\{0,1,2), \{3,4,5\}, \{6,7,8\},\{9,10,11\} \}$,  $\{\{0,4,8\}, \{3,7,11\}, \{2,6,10\},\{1,5,9\} \}$,\\
$\{\{0,7,10\}, \{3,6,9\}, \{2,5,8\}, \{1,4,11\} \}$.
 \end{itemize}
Finally, we apply Theorem \ref{theorem1} to the case $(1,3,3)$ and infer that  a $(K_2, P_3, K_{3})$-URD$(12;1,6,1)$ exists, too.
\end{proof}

\section{The spectrum for $\cF=\{K_2,P_4,C_4\}$}
\label{sec:MPC}

In this section we consider complex uniformly resolvable decompositions of the complete graph $K_v$ into $m$ parallel $1$-factors, $p$ parallel classes of $4$-paths, and  $c$ parallel classes of $4$-cycles.
The current problem is to determine the set of feasible
 triples $(m,p,c)$ such that $m\cdot p\cdot c \ne 0$, for which
 there exists a complex $(K_2,P_3,C_4)$-URD$(v;m,p,c)$.
The case of $\{P_4, C_4\}$, that is $m=0$, will be discussed in Section \ref{sec:PC}.

\begin{theorem}
 \label{theorem12}
 The necessary and sufficient conditions for the existence of a
  complex\/ $(K_2,P_4,C_4)$-URD\/$(v;m,p,c)$ are:
  \begin{description}
   \item[$(i)$] $v\ge 8$ and\/ $v$ is a multiple of\/ $4$;
   \item[$(ii)$] $2m + 3p + 4c = 2v-2$.
  \end{description}
Moreover, the parameters\/ $m,p,c$ are in the following ranges:
   \begin{description}
   \item[$(iii)$] $1\le c\le \frac v2 - 3$,

    \qquad \ $1\le m\le v-6$,

    \qquad \ $2\le p\le 2\, \lfloor \frac {v-4}3 \rfloor$;
   \item[$(iv)$]
    $p$ is even; and if\/ $p\equiv 2 ~ (\mbox{\rm mod } 4)$, then also\/
    $m$ is even.
   \end{description}

\end{theorem}

\begin{proof}
 We first show that the conditions are necessary.
Since $K_v$ has a $C_4$-factor, $v$ must be a multiple of~$4$. 
Further, as a $C_4$-, $K_2$-, and $P_4$-factor respectively cover exactly $v$, $v/2$, and $3v/4$ edges, in a $(K_2,P_4,C_4)$-URD$(v;m,p,c)$ we have
 $$
   \frac{mv}2 + \frac {3pv}4 + {cv} = {v\choose 2} .
 $$
This equality directly implies $(ii)$ and we may also conclude that $p$ is even and, further, if $p\equiv 2 ~ (\mbox{\rm mod } 4)$, then $m$ must be even as well.
Putting $p=2x$ we obtain $$m+3x+2c=v-1.$$
By our condition, all the three types of parallel classes are present in the decomposition, i.e.\ we have $c\ge 1$, $m \ge 1$, and $x \ge 1$. These, together with the equality above, imply the necessity of  $(iii)$.

\medskip

 Next we prove the sufficiency of $(i)$--$(ii)$.
We first take a 1-factorization of $K_{v/2}$ into $v/2-1$
 perfect matchings, which exists because $v$ is a multiple of 4.
Now, replace each vertex of $K_{v/2}$ with two non-adjacent vertices. This blow-up results in  $v/2-1$ parallel classes of $C_4$ inside $K_v$, and
 the missing edges can be taken as a perfect matching.
Let $C$ be the set of the parallel classes of $C_4$ and $x$ be a nonnegative integer such that
$0<x\leq \lfloor\frac{v-4}{3}\rfloor$. The construction splits into two cases depending on the parity of $x$.

If $x$ is even, take $\frac{3x}{2}$ parallel classes from $C$. Applying Theorem \ref{theorem4}, we transform the $\frac{3x}{2}$
parallel classes of $C_4$ into $2x=p$ parallel classes of paths $P_4$. For any given $y=c$ in the range $0<y \leq \lfloor\frac{v-2}{2}-\frac{3x}{2}\rfloor$,
keep $y$ classes of $C_4$ and transform the remaining $\frac{v-2}{2}-\frac{3x}{2}-y$ classes of $C_4$ into $2(\frac{v-2}{2}-\frac{3x}{2}-y)=m-1$ classes of 1-factors. In this way we obtain a complex $(K_2,P_4,C_4)$-URD$(v; v-3x-2y-1,2x,y)$.

If $x$ is odd, take $\frac{3(x-1)}{2}+2$ parallel classes from $C$. By Theorems~\ref{theorem2} and \ref{theorem4}, we can transform the $\frac{3(x-1)}{2}+2$ parallel classes of $C_4$ into $2x=p$ parallel classes of paths $P_4$ and a 1-factor.
For any given $y=c$ in the range $0<y \leq \frac{v-2}{2}-\frac{3(x-1)}{2}-2$, keep $y$ classes of $C_4$ and transform the remaining $\frac{v-2}{2}-\frac{3(x-1)}{2}-2-y$ classes of $C_4$ into $2(\frac{v-2}{2}-\frac{3(x-1)}{2}-2-y)=m-2$ classes of 1-factors. In this way, we obtain a complex $(K_2,P_4,C_4)$-URD$(v; v-3x-2y-1,2x,y)$.

The result, for every  $v\equiv 0\pmod{4}$, $0< x\leq \lfloor\frac{v-4}{3}\rfloor$, and $0< y\leq \lfloor\frac{v-2}{2}-\frac{3x}{2}\rfloor$, is a uniformly resolvable decomposition of $K_v$ into into $v-1-3x-2y=m$ classes containing only copies
of 1-factors, $2x=p$ classes containing only copies of paths $P_4$, and  $y=c$ classes containing only copies of 4-cycles $C_4$.
This finishes the proof of the theorem.
\end{proof}

\section{The spectrum for $\cH=\{P_4,C_4\}$}
\label{sec:PC}

Finally, we consider complex uniformly resolvable decompositions of the complete graph $K_v$ into $p$ classes containing only copies of paths $P_4$ and  $c$ classes containing only copies of 4-cycles $C_4$.
\begin{theorem}
 \label{theorem13}
 The necessary and sufficient conditions for the existence of a
  complex\/ $(P_4,C_4)$-URD\/$(v;p,c)$ are:
  \begin{description}
   \item[$(i)$] $v\ge 8$ and\/ $v$ is a multiple of\/ $4$;
   \item[$(ii)$] $3p+4c = 2v-2$.
  \end{description}
\end{theorem}

\begin{proof}
Necessity is a consequence of Theorem \ref{theorem12}, since we did not need to assume $m>0$ in that part of its proof.
Turning to sufficiency,
the condition $3p+4c=2v-2$ implies that $3p\equiv 2v-2\pmod{4}$. This gives $ p=2+4x$ and $c=\frac{v-4}{2}-3x$. 
For a construction, we start with a decomposition of $K_v$ into a perfect matching $F$ and $\frac{v-2}{2}$ parallel classes of $C_4$ as in the proof of Theorem \ref{theorem12}. By Theorem \ref{theorem3}, we can transform one class of $C_4$ and $F$ into two classes of paths $P_4$. Then, by Theorem \ref{theorem4}, we transform $3x$ parallel classes of $C_4$ into $4x$ parallel classes of $P_4$. The result, for every $x$ such that $0\leq x\leq \lfloor\frac{v-6}{6}\rfloor$,  is a uniformly resolvable decomposition of $K_v$ into $2+4x$ classes containing only copies
of paths $P_4$ and $\frac{v-4}{2}-3x$  classes containing only copies of 4-cycles $C_4$.
This completes the proof.
\end{proof}

\section{Conclusion}
Combining Theorems \ref{theorem11}, \ref{theorem12}, and \ref{theorem13}, we obtain the main result of this paper.

\begin{theorem}
 \label{theorem14}  ~~~
\begin{description}
   \item[$(i)$] A complex\/ $(K_2, P_3, K_3)$-URD\/$(v; m,p,t)$  exists if and only if\/ $v\ge 12$,\/ $v$ is a multiple of\/ $6$, and\/ $3m + 4p + 6t = 3v-3$.
   \item[$(ii)$] A complex\/ $(K_2, P_4, C_4)$-URD\/$(v; m,p,t)$ exists if and only if\/ $v\geq 8$,\/ $v$ is a multiple of\/ $4$, and\/ $2m + 3p + 4t = 2v-2$.
   \item[$(iii)$] A complex\/ $(P_4, C_4)$-URD\/$(v; p,t)$ exists if and only if\/ $v\geq 8$,\/ $v$ is a multiple of\/ $4$, and\/ $3p + 4t = 2v-2$.
 \end{description}
\end{theorem}

\medskip
Concerning the local metamorphosis studied in Section~\ref{sec:2}, we pose the following conjecture as a common generalization of Theorems~\ref{theorem1} and \ref{theorem4}.

\begin{conjecture}
  \label{conjecture} The union of\/ $k-1$ parallel classes of\/ $C_k$ is
  decomposable into\/ $k$ parallel classes of\/ $P_k$.
\end{conjecture}

\end{document}